\documentclass[a4paper, 12pt]{amsart}
\usepackage{amssymb}
\usepackage{amsmath}
\usepackage{amsthm}
\usepackage[dvipdfm]{graphicx} 
\usepackage{verbatim,enumerate}
 \newtheorem{theorem}{Theorem}[section]
 \newtheorem*{theorem*}{Theorem}
 \newtheorem*{lemma*}{Lemma}
 \newtheorem{proposition}[theorem]{Proposition}
 
 \newtheorem{fact*}{Fact}
 \newtheorem{lemma}[theorem]{Lemma}
 
\theoremstyle{definition}
 \newtheorem{definition}[theorem]{Definition}
 \newtheorem{remark}[theorem]{Remark}
 \newtheorem*{remark*}{Remark}
 
 \newtheorem{example}[theorem]{Example}
\numberwithin{equation}{section}

\usepackage[usenames]{color}

\newcommand{\vect}[1]{\boldsymbol{#1}}

\newcommand{\R}{\boldsymbol{R}}

\newcommand{\rank}{\operatorname{rank}}
\newcommand{\Hess}{\operatorname{Hess}}
\newcommand{\image}{\operatorname{image}}
\renewcommand{\phi}{\varphi}
\newcommand{\sign}{\operatorname{sgn}}
\newcommand{\inner}[2]{\left\langle{#1},{#2}\right\rangle}

\newcommand{\E}{\mathcal{E}}
\newcommand{\ep}{\varepsilon}
\newcommand{\zv}{\vect{0}}

\newcommand{\e}{\vect{e}}

\newcommand{\x}{\vect{x}}

\newcommand{\pmt}[1]{{\begin{pmatrix} #1  \end{pmatrix}}}

\newcommand{\resp}[1]{$($resp. {#1}$)$}
\newcommand{\mycomment}[1]{}
\begin{document}
\begin{center}
{\large 
{\bf Criteria for $D_4$ singularities of wave fronts}}
\\[5mm]
\today
\\[5mm]
Kentaro Saji
\\[7mm]
{\bf Abstract}
\end{center}
\begin{quote}
{\small We give useful and simple
criteria for determining $D_4^\pm$ singularities
of wave fronts.
As an application,
we investigate behaviors of singular curvatures of cuspidal edges
near $D_4^+$ singularities.}
\end{quote}
\section{Introduction}
In this paper,
we study criteria
for determining $D_4^\pm$ singularities.

A generic classification of
singularities of wave fronts
was given by
Arnol'd
and Zakalyukin.
They showed that the generic singularities of 
wave fronts in $\R^3$
are cuspidal edges and swallowtails.
Moreover, they showed that the generic singularities of
one-parameter bifurcation of wave fronts are
cuspidal lips, cuspidal beaks, cuspidal butterflies and
$D_4^\pm$ singularities (see \cite{AGV}).
Classifications of further degenerate singularities
have been considered by many authors 
(see \cite{AGV,arn-normal, arn-morse,garay,zakalag} for example).

To state the next theorem,
we define some terms here.
The unit cotangent bundle $T^*_1\R^{n+1}$ of  $\R^{n+1}$
has a canonical contact structure
and can be identified with the unit tangent bundle
$T_1\R^{n+1}$.
Let $\alpha$ denote the canonical contact form on it.
A map $i:M\to T_1\R^{n+1}$ is said to be
{\em isotropic\/} if
$\dim M=n$ and
the pull-back $i^*\alpha$ vanishes identically.
An isotropic immersion is called a {\em Legendrian immersion\/}.
We call 
the image of $\pi\circ i$
the {\em wave front set\/}
 of $i$,
where $\pi:T_1\R^{n+1}\to\R^{n+1}$ is the canonical projection.
We denote by $W(i)$ the wave front set of $i$.
Moreover, $i$ is called the {\em Legendrian lift\/} of $W(i)$.
With this framework, we define the notion of fronts as follows:
 A map-germ $f:(\R^n,\zv) \to (\R^{n+1},\zv)$ is called a
 {\em wave front\/} or a {\em front\/}
 if there exists a unit vector field $\nu$ of $\R^{n+1}$ along $f$
 such that
  $L=(f,\nu):(\R^n,\zv)\to (T_1\R^{n+1},\zv)$ is
 a Legendrian immersion (cf. \cite{AGV}, see also \cite{krsuy}).

The main result of this paper is as follows:
\begin{theorem}
 \label{thm:r3}
 Let\/ $f(u,v):(\R^2,\zv)\to(\R^3,\zv)$ be a front and\/
 $(f,\nu)$ its Legendrian lift. The germ\/ $f$ at\/ $\zv$ is 
 a $D_4^+$ singularity
 $($resp. $D_4^-$ singularity\/$)$ if and only if the following\/
 two conditions hold.
 \begin{itemize}
 \item[(a)] The rank of the differential 
            map $df_{\zv}$ is equal to zero.
 \item[(b)] $\det(\Hess
            \lambda(\zv))<0$ \resp{$\det(\Hess\lambda(\zv))>0$},
 \end{itemize}
 where\/ 
 $\lambda(u,v)
 =
 \det(f_u,f_v,\nu)$,
 $f_u=df(\partial/\partial u)$, $f_v=df(\partial/\partial v)$
 and\/ $\det(\Hess\lambda(\zv))$ means the determinant of
 the Hessian matrix of\/ $\lambda$ at\/ $\zv$.
\end{theorem}
Since our criteria 
require only the Taylor coefficients
of the given germ, 
this can be
useful for identifying the $D_4^\pm$ singularities
on explicitly parameterized maps.

Criteria for other singularities of fronts
were obtained in \cite{fsuy,is,horo,krsuy,suy3}.
Recently, several applications of
these criteria were obtained in various situations
\cite{ishi-machi,horo,circular,kruy,mu,sch,front}.

This paper is organized as follows.
In Section 2 we give fundamental notions
and state criteria for $4$-dimensional $D_4^\pm$ singularities
(Theorem \ref{thm:r4}).
In Section 3 we prove Theorem \ref{thm:r4}, and
in Section 4 we prove Theorem \ref{thm:r3}.
In Section 5 we apply Theorem \ref{thm:r3}
to the normal forms of the $D_4^\pm$ singularities,
which
confirms one direction of the conclusion of Theorem\/ $\ref{thm:r3}$,
since the conditions in Theorem\/ $\ref{thm:r3}$
are independent of the right-left equivalence.
In Section 6 as an application of Theorem \ref{thm:r3},
we study the singular curvatures of four cuspidal edges near
a ``generic'' $D_4^+$ singularity.

The author would like to express his sincere gratitude
to Toshizumi Fukui,
Peter Giblin, Shyuichi Izumiya, Toru Ohmoto, Wayne Rossman
and Masaaki Umehara
for fruitful discussions and helpful comments.

All maps considered here are of class $C^\infty$.

\section{Fundamental notions}
\label{sec:pre}
Let $f(u_1\ldots,u_n):(\R^n,\zv)\to(\R^{n+1},\zv)$ be a front and
$L=(f,\nu):(\R^n,\zv)\to (T_1\R^{n+1},\zv)$
its Legendrian lift.
The isotropicity of $L$ is
equivalent to the orthogonality condition
\begin{equation*}
 \label{eq:orthogonal}
 \inner{df(X_p)}{\nu(p)}=0 \qquad
  \left(X_p\in T_p\R^n,\  p\in\R^n\right),
\end{equation*}
where $\inner{~}{~}$ is the Euclidean inner product.
The vector field  $\nu$ is called the {\em unit normal vector field\/} of
the front $f$.
For a front
$f:(\R^n,\zv)\to(\R^{n+1},\zv)$,
a function
\begin{equation}
 \label{eq:signed-area-density}
   \lambda(u_1,\ldots,u_n)=
   \det(f_{u_1},\ldots,f_{u_n},\nu)(u_1,\ldots,u_n)
\end{equation}
is called the {\em signed volume density function\/} of $f$,
where $f_{u_i}=df(\partial/\partial u_i)$, $(i=1,\ldots,n)$.
The set of singular points $S(f)$ of $f$ coincides with
the zeros of $\lambda$.
If $n=3$ and the rank of $df_{\zv}$
is equal to $1$, then there exist vector fields
$\tau, \xi, \eta$ such that 
$\xi_{\zv}$ and\/ $\eta_{\zv}$ generate the kernel of\/
             $df_{\zv}$, and\/
             $\tau_{\zv}$ is transverse to\/ $\ker(df_{\zv})$.
\begin{definition}
 \label{def:eq}
 Two map-germs $f_1,f_2:(\R^n,\zv)\to(\R^m,\zv)$
 are {\em right-left equivalent}\/
 if
 there exist diffeomorphisms $S:(\R^n,\zv)\to(\R^n,\zv)$
 and $T:(\R^m,\zv)\to(\R^m,\zv)$
 such that
 $
 f_2\circ S=T\circ f_1
 $
 holds. If one can take $T$ to be the identity,
 the two map-germs are called {\em right equivalent}.
\end{definition}
\begin{definition}
 \label{def:sing}
 \mycomment{A {\em cuspidal edge\/} is a
 map-germ right-left equivalent to
 $(u,v)\mapsto$\\$(u,v^2,v^3)$ at $\zv$.}
 \hfill A\hfill  {\em cuspidal\hfill edge\/}\hfill is\hfill a\hfill
 map-germ\hfill right-left\hfill equivalent\hfill to\hfill
 $(u,v)\mapsto$\\$(u,v^2,v^3)$ at $\zv$.
 A {\em swallowtail\/} is a map-germ right-left equivalent to
 $(u,v)\mapsto(u,3v^4+uv^2,4v^3+2uv)$ at $\zv$.
 A map-germ right-left equivalent to
 $(u,v)\mapsto(uv, u^2+3\ep v^2,u^2v+\ep v^3)$
 at $\zv$ is called a $D_4^+$ {\em singularity\/}
 if $\ep=1$
 (resp. a $D_4^-$ {\em singularity\/} if $\ep=-1$)
 (see Figure \ref{fig:dfour},
 where the left-hand figure is the $D_4^+$ singularity
 and the right-hand figure is the $D_4^-$ singualrity).
 A map-germ $(\R^3,\zv)\to(\R^4,\zv)$ right-left 
 equivalent to
 $(u,v,t)\mapsto(uv, u^2+2tv\pm 3v^2,2u^2v+tv^2\pm 2v^3,t)$
 at $\zv$ is called a
 {\em $4$-dimensional\/ $D_4^\pm$ singularity\/}, respectively.
\end{definition}
\begin{figure}[htbp]
 \centering
 \begin{tabular}{ccc}
  \includegraphics[width=.25\linewidth,bb=14 14 361 367]{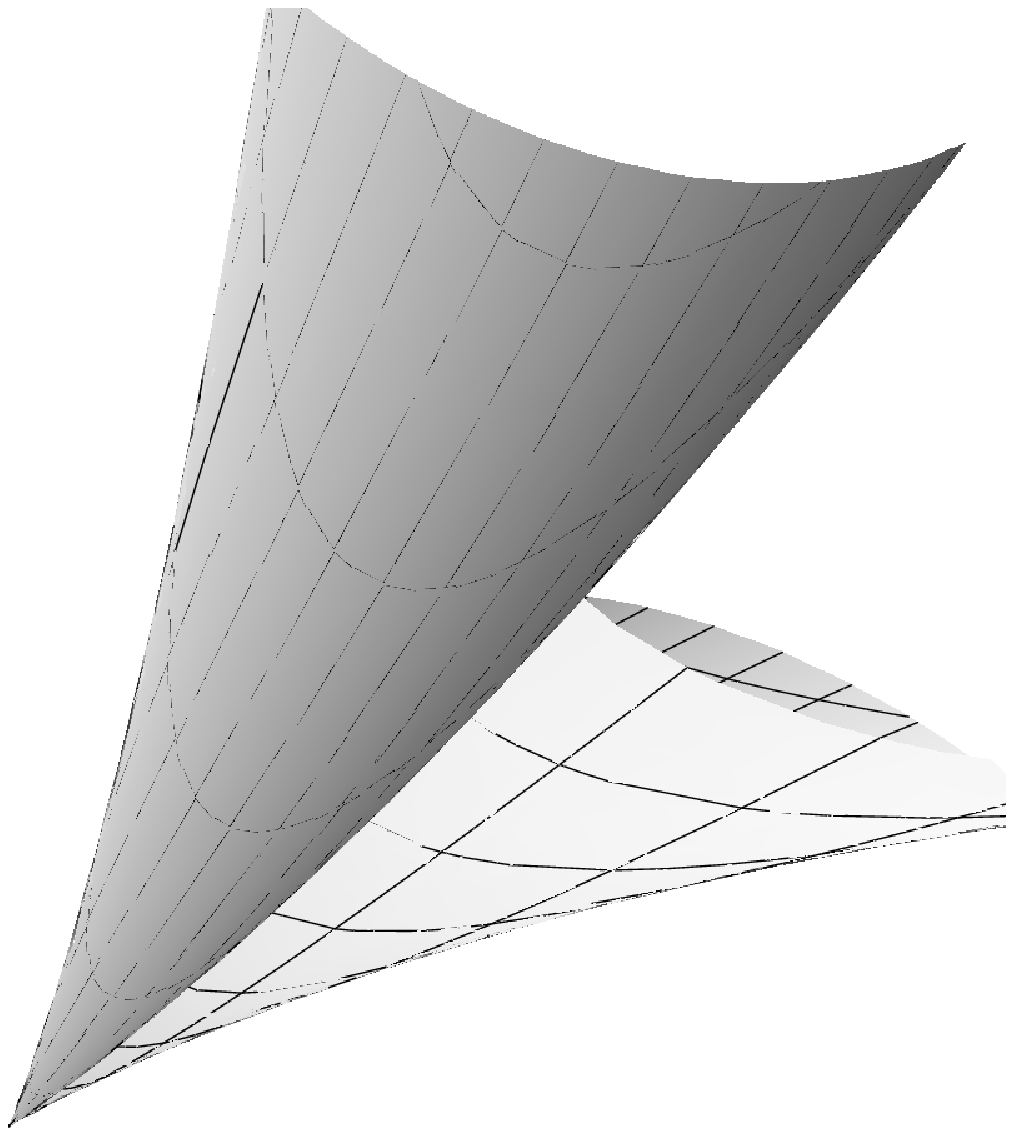}&
  \hspace{30mm}&
  \includegraphics[width=.25\linewidth,bb=14 14 249 225]{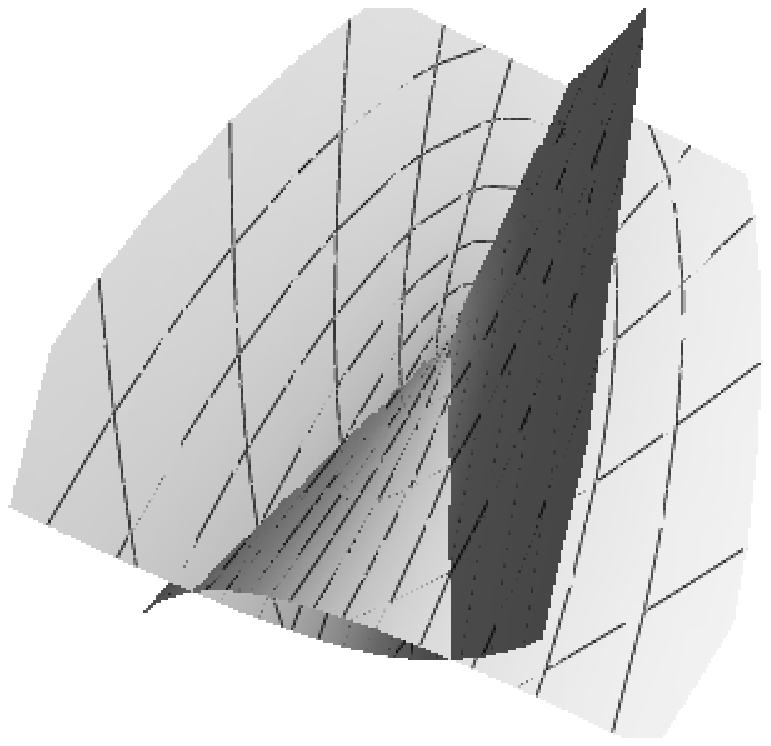}
 \end{tabular}
 \caption{The $D_4^\pm$ singularities}
 \label{fig:dfour}
\end{figure}

Since $D_4^\pm$ singularities
 appear as generic singularities of fronts
$(\R^3,\zv)\to(\R^4,\zv)$,
Theorem \ref{thm:r3} is based on the following theorem:
\begin{theorem}
 \label{thm:r4}
 Let\/ $f(u,v,t):(\R^3,\zv)\to(\R^4,\zv)$ be a front and\/
 $\nu$ its unit normal vector.
 The germ\/ $f$ at\/ $\zv$ is a\/ $4$-dimensional\/ 
 $D_4^+$ singularity $($resp. $4$-dimensional\/ $D_4^-$ singualrity$)$
 if and only if the following
 three conditions hold.
 \begin{itemize}
  \item[(1)] The rank of the differential 
            map\/ $df_{\zv}$ is equal to one.
  \item[(2)] $\det(\Hess_{(\xi,\eta)}
            \lambda(\zv))<0$ \resp{$\det(\Hess_{(\xi,\eta)}
            \lambda(\zv))>0$}, 
            where\/
            the\/ $2\times 2$ matrix\/
            $\Hess_{(\xi,\eta)}\lambda$ is the Hessian matrix with respect to\/
            ${\xi}$ and\/ ${\eta}$ of\/ $\lambda$.
            Here, $\xi$ and $\eta$ are vector fields on $\R^3$
            that generate the kernel of\/
             $df_{\zv}$.
 \item[(3)] The map-germ\/ 
            $$\big(\inner{df(\xi)}{d\nu(\xi)},
             \inner{df(\xi)}{d\nu(\eta)},
             \inner{df(\eta)}{d\nu(\eta)}\big)
            :(\R^3,\zv)\to(\R^3,\zv)$$ is an immersion at\/ $\zv$.
 \end{itemize}
 \end{theorem}
Note that 
$\Hess_{(\xi,\eta)}\lambda(\zv)$ is symmetric,
since $\xi$ and $\eta$ belong to the kernel of $df_{\zv}$.
\begin{remark}
 In Theorem \ref{thm:r3},
 the condition corresponding to $(3)$ in Theorem \ref{thm:r4}
 is  not needed, and is
 not used in the theorem's proof.
 On the other hand, if a front $f:(\R^2,\zv)\to(\R^3,\zv)$
 satisfies the conditions $(a)$ and $(b)$ of Theorem \ref{thm:r3},
 then the map 
 $(h_{11},h_{12},h_{22}):(\R^2,\zv)\to(\R^3,\zv)$ is an immersion,
 where $h_{11},h_{12},h_{22}$ are the components of
 the second fundamental form of $f$.
 See \cite{suyojm} for relations between $A_k$ singularities
 and the second fundamental form of fronts.
 \mycomment{in n=2 case, (L,M,N) is not immersion,
 then $f_uu,f_uv,f_vv$ are all parallel.
 Then $\hess\lambda=O$.}
\end{remark}

\section{Proof of Theorem \ref{thm:r4}}
In this section, we prove Theorem \ref{thm:r4}.
First, we show that the conditions in Theorem \ref{thm:r4}
do not depend on the choices of vector fields and the coordinate systems.
\subsection{Well-definedness of the conditions}
It is easy to see that
the conditions in Theorem \ref{thm:r4} are
independent of the choices of vector fields
on the source.
We show that they are independent of the choice of
coordinate system on the target as well.

We take a diffeomorphism $T:(\R^4,\zv)\to(\R^4,\zv)$.
The differential map $dT$ can be considered as a $GL(4,\R)$-valued
function $q\mapsto dT_{q}$.
Since $A\x_1\wedge A\x_2 \wedge A\x_3
=
(\det A)\,{}^t{A^{-1}}(\x_1\wedge \x_2\wedge \x_3)$ 
holds for any
vectors $\x_1,\x_2,\x_3\in\R^4$ and any non-singular
matrix $A$, we can take
\begin{equation}
\label{eq:nutrans}
\tilde\nu=\frac{1}{\delta}\,{}^t(dT_{f})^{-1}\nu,\qquad
\delta=\sqrt{\inner{{}^t(dT_f)^{-1}\nu}{{}^t(dT_f)^{-1}\nu}}
\end{equation}
as the unit normal vector of $T\circ f$.
Using \eqref{eq:nutrans},
we can easily see that
the conditions (2) and (3) of Theorem \ref{thm:r4} are
independent of the choice of coordinate system on
the target by noticing that $(dT_f)_u=(dT_f)_v=O$ holds at $\zv$
if $df_{\zv}(\partial/\partial u)=
df_{\zv}(\partial/\partial v)=\zv$.
\subsection{
Criteria for a function to be right equivalent to $u^3\pm uv^2$}
\label{sec:criteria-d4}
\mycomment{In the criteria for $A_k$ singularities of fronts,
the criteria for $A_k$-type singularities of functions
play the crucial role.}
In order to show Theorem \ref{thm:r4},
criteria for a function to be right equivalent
to the function $u^3\pm uv^2$
play the crucial role.
For a function-germ $\phi:(\R^2,\zv)\to(\R,0)$,
we define the following number:
\begin{equation}
\label{eq:cubdiscri}
\begin{array}{l}
    \Delta_{\phi}=
    \big((\phi_{uuu})^2(\phi_{vvv})^2
    -6\phi_{uuu}\phi_{uuv}\phi_{uvv}\phi_{vvv}
    -3(\phi_{uuv})^2(\phi_{uvv})^2\\
    \hspace{70mm}
     +4(\phi_{uuv})^3\phi_{vvv}
   +4\phi_{uuu}(\phi_{uvv})^3\big)(0,0).
\end{array}
\end{equation}
Then the following lemma holds:
\mycomment{
This function gives us the conditions that 
a function is right equivalent
to the function $u^3\pm uv^2$ as follows.}
\begin{lemma}
 \label{lem:funcd4}
 Let\/ $\phi:(\R^2,\zv)\to(\R,0)$ be a function satisfying\/
  $j^2\phi=0$.
 Then\/ $\phi$ at\/ $\zv$ is right equivalent to\/ $(u^3+uv^2)$
 $($resp. $(u^3-uv^2))$
 if and only if\/ $\Delta_{\phi}>0$ 
 \resp{$\Delta_{\phi}<0$}.
\end{lemma}
 Here, $j^k\phi$ means the $k$-jet of $\phi$ at $\zv$.
\begin{proof}
 The number $\Delta_{\phi}$ is the discriminant of
 the cubic equation $j^3\phi(u,1)=0$ with respect to $u$. 
 A standard method yields that $\Delta_{\phi}>0$ 
 \resp{$\Delta_{\phi}<0$} is the necessary and
 sufficient condition for $j^3\phi$ to be
 right equivalent to $u^3+ uv^2$ \resp{$u^3- uv^2$}.
 It is known that if $j^3\phi$
 is right equivalent to $u^3\pm uv^2$,
 then $\phi$ is right equivalent to $u^3\pm uv^2$.
 Thus we have the lemma.
\end{proof}
\mycomment{Fukui and Nu\~no-Ballesteros also obtained the same condition
that a function which is
 right equivalent to $u^3\pm uv^2$ (\cite{nbf}).}

\subsection{Versal unfoldings and their discriminant sets}
Let us define a function
\begin{equation}
\label{eq:typversal}
{\mathcal V}(u,v,x,y,z,t)=u^3\pm uv^2+u^2t+ux+vy+z.
\end{equation}
Then ${\mathcal V}$ is a ${\mathcal K}$-versal unfolding
of ${\mathcal V}(u,v,0,0,0,0)=u^3\pm uv^2$.
By definition, any ${\mathcal K}$-versal unfolding
of a function which is right equivalent to 
$u^3\pm uv^2$ is $P$-${\mathcal K}$-equivalent to
${\mathcal V}$. 
See \cite[Section 8]{AGV} for 
the definitions of unfoldings,
their ${\mathcal K}$-versality
and
$P$-${\mathcal K}$-equivalence between them.
See also \cite[Section 7]{horo}.
An unfolding $G(u,v,\x):(\R^2\times \R^k,(\zv,\zv))\to(\R,0)$
of a function $g:(\R^2,\zv)\to(\R,0)$
is a {\em Morse family of hypersurfaces\/}
if the map
$(G,G_u,G_v)
:
(\R^2\times\R^k,(\zv,\zv))\to(\R^3,\zv)$
is a submersion.
Moreover, the {\em discriminant set\/} ${\mathcal D}_{G}$
of a Morse family of hypersurfaces
$G$ is defined by
$$
{\mathcal D}_G=\{\x\in\R^k;
\text{\ there\ exists\ }(u,v)\in\R^2\text{\ such\ that\ }
G=G_u=G_v=0\text{\ at\ }(u,v,\x)\}.
$$
If two Morse families of hypersurfaces
$G_1, G_2 : (\R^2\times\R^k,\zv) \to  (\R,0)$
satisfy
that both regular sets
of their discriminant sets
are dense in $(\R^k,\zv)$, then
$G_1$ and $G_2$  are  $P$-${\mathcal K}$-equivalent
if and only if
the discriminant sets
$({\mathcal D}_{G_1},\zv)$ and $({\mathcal D}_{G_2},\zv)$ 
are diffeomorphic as set-germs
(see \cite[Section 1.1]{zaka},
see also \cite[Appendix A]{krsuy}).

The discriminant set
of ${\mathcal V}$ is given by 
$$
{\mathcal D}_{{\mathcal V}}=
\{(x,y,z,t);
x=-3u^2\mp v^2-2ut,\,
y=\mp 2uv,\,
z=2u^3+u^2t\pm2uv^2\}.
$$
This gives 
a parameterization of a $4$-dimensional $D_4^\pm$ singularity.

Thus, in order to show Theorem \ref{thm:r4},
it is sufficient to construct a function $\phi$ and an 
unfolding $\Phi$ of $\phi$,
which is a Morse family of hypersurfaces,
such that the discriminant set coincides with
the image of $f$, and
show that
$\phi$ is right equivalent to $u^3\pm uv^2$
and $\Phi$ is a ${\mathcal K}$-versal unfolding.

\subsection{Unfolding of a given front}
Let $f:(\R^3,\zv)\to(\R^4,\zv)$ be a front and $\nu$ its unit
normal vector.
We assume that
conditions (1), (2) and (3) of Theorem \ref{thm:r4}  are satisfied.
In particular,
$\rank df_{\zv}=1$ holds,
so by the implicit function theorem and
by coordinate transformations on the source and target,
$f$ can be written as
$$
f(u,v,t)=(f_1(u,v,t),f_2(u,v,t),f_3(u,v,t),t)
,\quad d(f_1,f_2,f_3)_{\zv}=\zv,
$$
where $\partial/\partial u$ and $\partial/\partial v$
generate the kernel of $df_{\zv}$,
and $\nu(\zv),\nu_u(\zv),\nu_v(\zv)$ are orthonormal.
Now we define functions
$\Phi$ and $\varphi$ by
\begin{equation*}
 \label{eq:phi}
\begin{array}{rcl}
 \Phi(u,v,x,y,z,t)
 &=&
 \inner{
 \big(f_1(u,v,t),f_2(u,v,t),f_3(u,v,t),t\big)
 -
 (x,y,z,t)
 }
 {\nu(u,v,t)}\\[1mm]
 &=&
 \inner{\big(f_1-x,f_2-y,f_3-z,0\big)}{\nu},\\
\phi(u,v)&=&\Phi(u,v,0,0,0,0).
\end{array}
\end{equation*}
\begin{lemma}
 The discriminant set ${\mathcal D}_\Phi$ of\/ $\Phi$ coincides
 with
 the image of $f$.
\end{lemma}
\begin{proof}
 It is easy to check that
 the image of $f$ is contained in ${\mathcal D}_\Phi$.
 We now show $\image f\supset {\mathcal D}_\Phi$.
 Set
 $\vect{w}=(x,y,z,t)\in {\mathcal D}_\Phi$.
 Since $\Phi_u=\inner{f_u}{\nu}+\inner{f-\vect{w}}{\nu_u}
 =\inner{f-\vect{w}}{\nu_u}$
 holds,
 $\vect{w}\in {\mathcal D}_\Phi$ is equivalent to
 existence of $(u,v)$ such that
 $\inner{f(u,v,t)-\vect{w}}{\nu(u,v,t)}=0$,
 $\inner{f(u,v,t)-\vect{w}}{\nu_u(u,v,t)}=0$ and 
 $\inner{f(u,v,t)-\vect{w}}{\nu_v(u,v,t)}=0$.
 Moreover, since $\inner{f(u,v,t)-\vect{w}}{\e_4}=0$ and
 $f_t(\zv)=\e_4$ hold, the four vectors $\{\e_4,\nu,\nu_u,\nu_v\}$
 are linearly independent near $\zv$,
 where $\e_4=(0,0,0,1)$.
 Thus it follows that $f(u,v,t)-\vect{w}=\zv$ for some $(u,v)$.
 Hence we have $\vect{w}\in\image f$.
\end{proof}
Since we assume that $f$ is a front, $\Phi$ is a Morse family
of hypersurfaces.
By the assumption $\det(\Hess\lambda)\ne0$,
the regular points of $f$ are dense in $(\R^3,\zv)$.
\subsection{Right equivalence of $\phi$ and $u^3\pm uv^2$}
Let $f:(\R^3,\zv)\to(\R^4,\zv)$ be a front
and assume that the conditions $(1),(2)$ and $(3)$ of
Theorem \ref{thm:r4} are satisfied.
Let $\Phi$ and $\phi$ be as above. Here, we prove that
if $\det(\Hess\lambda(\zv))>0$ \resp{$\det(\Hess\lambda(\zv))<0$}
then
$\phi$ is right equivalent to $u^3- uv^2$ \resp{$u^3+uv^2$}.

Calculating the third order differentials of $\phi$,%
\mycomment{
Since
$$
\begin{array}{rcl}
 \phi_u&=&\inner{f_u}{\nu}+\inner{f-\vect{w}}{\nu_u}
  =\inner{f-\vect{w}}{\nu_u},\\
 \phi_{uu}&=&\inner{f_u}{\nu_u}+\inner{f-\vect{w}}{\nu_{uu}},\\
 \phi_{uuu}&=&\inner{f_{uu}}{\nu_u}+2\inner{f_u}{\nu_{uu}}+
  \inner{f-\vect{w}}{\nu_{uuu}},\\
 \phi_{uuv}&=&
  \inner{f_{uv}}{\nu_u}+\inner{f_u}{\nu_{uv}}
  +\inner{f_v}{\nu_{uu}}+\inner{f-\vect{w}}{\nu_{uuv}},\\
 \phi_{uv}&=&\inner{f_v}{\nu_u}+\inner{f-\vect{w}}{\nu_{uv}},\\
 \phi_v&=&\inner{f_v}{\nu}+\inner{f-\vect{w}}{\nu_v}=
  \inner{f-\vect{w}}{\nu_v},\\
 \phi_{vv}&=&\inner{f_v}{\nu_v}+\inner{f-\vect{w}}{\nu_{vv}},\\
 \phi_{uvv}&=&\inner{f_{uv}}{\nu_v}+\inner{f_v}{\nu_{uv}}+
  \inner{f_u}{\nu_{vv}}+\inner{f-\vect{w}}{\nu_{uuv}},\\
 \phi_{vvv}&=&\inner{f_{vv}}{\nu_v}+2\inner{f_v}{\nu_{vv}}
  +\inner{f-\vect{w}}{\nu_{vvv}},
\end{array}
$$
it holds that} we have
$$
\begin{array}{l}
 \phi_u=
 \phi_{uu}=
 \phi_{uv}=
 \phi_v=
 \phi_{vv}=0,\\
 \phi_{uuu}=\inner{f_{uu}}{\nu_u},\quad
 \phi_{uuv}=
  \inner{f_{uv}}{\nu_u},\ 
 \phi_{uvv}=\inner{f_{uv}}{\nu_v},\ 
 \phi_{vvv}=\inner{f_{vv}}{\nu_v}
\end{array}
$$
at $\zv$.
Also, $\inner{f_u}{\nu}=\inner{f_v}{\nu}=0$ holds
identically,
and taking derivatives of this, we have
\begin{equation}
 \label{eq:itti}
  \begin{array}{ll}
   \inner{f_{uuu}}{\nu}=-2\inner{f_{uu}}{\nu_u}&=-2\phi_{uuu},\\
   \inner{f_{uuv}}{\nu}=-2\inner{f_{uv}}{\nu_u}
    =-2\inner{f_{uu}}{\nu_v}&=-2\phi_{uuv},\\
   \inner{f_{uvv}}{\nu}=-2\inner{f_{uv}}{\nu_v}=
    -2\inner{f_{vv}}{\nu_u}&=-2\phi_{uvv},\\
   \inner{f_{vvv}}{\nu}=-2\inner{f_{vv}}{\nu_v}&=-2\phi_{vvv}
  \end{array}
\end{equation}
at $\zv$.
On the other hand,
$\{f_t,\nu_u,\nu_v,\nu\}$ are linearly independent near $\zv$,
so there exist functions $a_i,b_i,c_i$ $(i=1,2,3,4)$ such that
$$
\begin{array}{rcl}
 f_{uu}&=&a_1f_t+a_2\nu_u+a_3\nu_v+a_4\nu,\\
 f_{uv}&=&b_1f_t+b_2\nu_u+b_3\nu_v+b_4\nu,\\
 f_{uu}&=&c_1f_t+c_2\nu_u+c_3\nu_v+c_4\nu.
\end{array}
$$
Note that by (\ref{eq:itti}), $a_3(\zv)=b_2(\zv)$ and $b_3(\zv)=c_2(\zv)$ hold.
By a direct calculation,
\begin{equation}
\label{eq:lambdadd}
\lambda_{uu}=a_2b_3-a_3b_2,\ 
\lambda_{uv}=a_2c_3-a_3c_2,\ 
\lambda_{vv}=b_2c_3-b_3c_2
\end{equation}
and
\begin{equation}
\label{eq:atophi}
\begin{array}{l}
 a_2=\inner{f_{uu}}{\nu_u}=\phi_{uuu},\qquad
 a_3=b_2=\inner{f_{uu}}{\nu_v}=\phi_{uuv},\\
 \hspace{20mm}b_3=c_2=\inner{f_{uv}}{\nu_v}=\phi_{uvv},\qquad
 c_3=\inner{f_{vv}}{\nu_v}=\phi_{vvv}
\end{array}
\end{equation}
hold at $\zv$.
By 
\eqref{eq:cubdiscri}, \eqref{eq:itti}, \eqref{eq:lambdadd}
and \eqref{eq:atophi}, we have
$$
\det\big(
\Hess_{(\partial/\partial u,\partial/\partial v)}\lambda(\zv)\big)
=-
\Delta_{\phi}.
$$
This proves the assertion.
\subsection{Versality of $\Phi$}
Here, we prove that $\Phi$ is a ${\mathcal K}$-versal
unfolding of $\phi$.

Recall that an unfolding $\Phi(u,v,x,y,z,t)$ of $\phi(u,v)$
is ${\mathcal K}$-versal if
the following equality holds 
(\cite{marti}, see also \cite[Appendix]{izkos}):
\mycomment{
$$
{\mathcal E}_2
=
\left\langle\frac{\partial\phi}{\partial u},
 \frac{\partial\phi}{\partial v},\phi
 \right\rangle_{{\mathcal E}_2}
+V_\Phi,\quad
\left(
V_\Phi
=
\left\langle
\frac{\partial \Phi}{\partial X},
\frac{\partial \Phi}{\partial Y},
\frac{\partial \Phi}{\partial Z},
\frac{\partial \Phi}{\partial T}
\right\rangle_{\!\!\R}
{(u,v,0,0,0,0)}
\right).
$$
}
\begin{equation}
\label{eq:versalcond}
{\mathcal E}_2
=
\left\langle\phi_u,
 \phi_v,\phi
 \right\rangle_{{\mathcal E}_2}
+V_\Phi,
\end{equation}
where
${\mathcal E}_2$ is the local ring of function-germs
$(\R^2,\zv )\to \R$
with the unique maximal 
ideal
${\mathcal M}_2=\{h\in {\mathcal E}_2;h(0)=0\}$,
and
$\langle\phi_u,\phi_v,\phi\rangle_{{\mathcal E}_2}$
is
the ideal generated by $\phi_u$, $\phi_v$ and $\phi$
in ${\mathcal E}_2$.
Moreover, $V_{\Phi}$ is the vector subspace of $\E_2$
generated by
$\Phi_x(u,v,\zv)$,
$\Phi_y(u,v,\zv)$,
$\Phi_z(u,v,\zv)$ and
$\Phi_t(u,v,\zv)$ over $\R$.\mycomment{ i.e.,
$$
V_{\Phi}=
\left\langle
 \Phi_x(u,v,\zv),
\Phi_y(u,v,\zv),
\Phi_z(u,v,\zv),
\Phi_t(u,v,\zv)\right\rangle_{\R}.
$$}
We set $\nu=(\nu_1,\nu_2,\nu_3,\nu_4)$.
Then we see that
$
\Phi_x=\nu_1,
\Phi_y=\nu_2,
\Phi_z=\nu_3
$.
Since $f$ is a front,
$
df_{\zv}(\partial/\partial u)=
df_{\zv}(\partial/\partial v)=\zv
$
and $\nu_4(\zv)=\zv$, and
$$
\left\langle \nu_1(u,v,0),\nu_2(u,v,0),\nu_3(u,v,0)
\right\rangle_{\R}\supset
\R\oplus u\R\oplus v\R
$$
holds,
where
$\langle \nu_1,\nu_2,\nu_3\rangle_{\R}$
means the vector space generated by 
$\nu_1,\nu_2,\nu_3$ over $\R$.
On the other hand,
$\phi$ is right equivalent to $u^3\pm uv^2$.
It is known that if a function
$\phi$ is right equivalent to $u^3\pm uv^2$,
then
$
({\mathcal M}_2)^3\subset {\mathcal M}_2\left\langle
\phi_u,\phi_v\right\rangle_{{\mathcal E}_2}
$
is satisfied.
We set
$$\psi_1(u,v)=\Phi_t(u,v,\zv),\qquad
\psi_2(u,v)=\phi_u(u,v),\qquad
\psi_1(u,v)=\phi_v(u,v).$$
Since ${\mathcal M}_2\left\langle
\phi_u,\phi_v\right\rangle_{{\mathcal E}_2}
\subset \left\langle
\phi_u,\phi_v,\phi\right\rangle_{{\mathcal E}_2}
$ holds, for showing \eqref{eq:versalcond}
it suffices to prove that
$$
\left\langle
\psi_1(u,v),
\psi_2(u,v),
\psi_3(u,v)
\right\rangle_{\R}
\supset
u^2\R\oplus uv\R\oplus v^2\R.
$$
This is equivalent to
\begin{equation*}
 \label{eq:versal1}
  \det\pmt{
  (\psi_1)_{uu}&(\psi_1)_{uv}&(\psi_1)_{vv}\\
 (\psi_2)_{uu}&(\psi_2)_{uv}&(\psi_2)_{vv}\\
 (\psi_3)_{uu}&(\psi_3)_{uv}&(\psi_3)_{vv}\\
 }(\zv)
  =
  \det\pmt{
  (\psi_1)_{uu}&(\psi_1)_{uv}&(\psi_1)_{vv}\\
 \phi_{uuu}&\phi_{uuv}&\phi_{uvv}\\
 \phi_{uuv}&\phi_{uvv}&\phi_{vvv}\\
 }(\zv)
  \ne0.
\end{equation*}
Since
$
\psi_1(u,v)
=
-\nu_4(u,v,0)+\inner{f}{\nu_t}(u,v,0)
$ holds, we have
\begin{equation}
\label{eq:psi1}
\begin{array}{l}
 (\psi_1)_{uu}=-\inner{\nu_{uu}}{f_t}+\inner{f_{uu}}{\nu_t},\\
 (\psi_1)_{uv}=-\inner{\nu_{uv}}{f_t}+\inner{f_{uv}}{\nu_t},\\
 (\psi_1)_{vv}=-\inner{\nu_{vv}}{f_t}+\inner{f_{vv}}{\nu_t}
\end{array}
\end{equation}
at $\zv$.
Taking derivatives of
 $\inner{f_t}{\nu}=\inner{f_u}{\nu}=\inner{f_v}{\nu}=0$,
we have
\begin{equation}
\label{eq:ident}
\begin{array}{l}
 -\inner{\nu_{uu}}{f_t}+\inner{f_{uu}}{\nu_t}=\inner{f_{tu}}{\nu_u},\\
 -\inner{\nu_{uv}}{f_t}+\inner{f_{uv}}{\nu_t}=
  (1/2)(\inner{f_{tu}}{\nu_v}+\inner{f_{tv}}{\nu_u})
  =\inner{f_{tu}}{\nu_v}=\inner{f_{tv}}{\nu_u},\\
 -\inner{\nu_{vv}}{f_t}+\inner{f_{vv}}{\nu_t}=\inner{f_{tv}}{\nu_v}
\end{array}
\end{equation}
at $\zv$. Hence by \eqref{eq:psi1} and \eqref{eq:ident},
it holds that
$$
 \pmt{
 (\psi_1)_{uu}&(\psi_1)_{uv}&(\psi_1)_{vv}\\
 \phi_{uuu}&\phi_{uuv}&\phi_{uvv}\\
 \phi_{uuv}&\phi_{uvv}&\phi_{vvv}\\
 }(\zv)
 =
 \pmt{
 \inner{f_{tu}}{\nu_u}&\inner{f_{tu}}{\nu_v}&\inner{f_{tv}}{\nu_v}\\
 \inner{f_{uu}}{\nu_u}&\inner{f_{uv}}{\nu_u}&\inner{f_{uv}}{\nu_v}\\
 \inner{f_{uv}}{\nu_u}&\inner{f_{uv}}{\nu_v}&\inner{f_{vv}}{\nu_v}
 }(\zv).
$$
This is the Jacobi matrix of
the map
$$
\big(\inner{f_u}{\nu_u},
\inner{f_u}{\nu_v},
\inner{f_v}{\nu_v}\big)
(t,u,v):\R^3\to\R^3.
$$
This proves the assertion.

\section{Proof of Theorem \ref{thm:r3}}
To prove criteria for the three dimensional case,
we first show the following lemma.
\begin{lemma}
 \label{lem:key} 
 Let\/
 $
 G(u,v,x,y,z):(\R^2\times\R^3,(\zv,\zv))\to(\R,0)
 $
 be an unfolding of a function\/
 $u^3+uv^2$ \resp{$u^3-uv^2$}.
 Suppose that\/ $G$ is a Morse family of hypersurfaces and 
 that the regular set of its discriminant set is dense in\/ $(\R^2,\zv)$.
 Then\/ $G$ is\/ $P$-${\mathcal K}$ equivalent to
 $$
 {\mathcal V}_0(u,v,x,y,z)=u^3+uv^2+xu+yv+z,\quad
 (\text{resp.}\ {\mathcal V}_0=u^3- uv^2+xu+yv+z).
 $$
\end{lemma}
\begin{proof}
 Since the unfolding
 ${\mathcal V}$ as in \eqref{eq:typversal}
  is a ${\mathcal K}$-versal unfolding of $u^3\pm uv^2$,
 there exists a map $(g_1,g_2,g_3,g_4):\R^3\to\R^4$ such that
 $G$ is $P$-${\mathcal K}$ equivalent to
 $$
 G_1=
 u^3\pm uv^2+g_1(x,y,z)u^2+g_2(x,y,z)u+g_3(x,y,z)v+g_4(x,y,z).
 $$
 Since the condition and the assertion of the lemma do not depend
 on the $P$-${\mathcal K}$ equivalence,
 we may suppose that $G$ is equal to $G_1$.
 Moreover, 
 $(g_2,g_3,g_4):\R^3\to\R^3$ is an immersion,
 because $G$ is 
 a Morse family of hypersurfaces.
 Thus $G$ is $P$-${\mathcal K}$ equivalent to 
 $$G_2=
 u^3\pm uv^2+g_1(x,y,z)u^2+xu+yv+z.
 $$
 Hence we may suppose $G$ is equal to $G_2$.
 Now we consider the following function
 $$
 \overline{G}(u,v,t,x,y,z)=u^3\pm uv^2+(t-g_1(x,y,z))u^2+xu+yv+z.
 $$
 The following 
 is the special case of Zakalyukin's lemma \cite[Theorem 1.4]{zaka}.
 Note that we are considering the ${\mathcal K}$-versal unfolding
  ${\mathcal V}$ as in \eqref{eq:typversal},
 the newtral subspace 
 of ${\mathcal V}$ is empty (see  \cite[Section 1.3]{zaka}).
 \begin{lemma}
  \label{lem:zakaone}
  For the\/ ${\mathcal K}$-versal unfolding\/
  ${\mathcal V}:(\R^2\times\R^4,\zv)\to(\R,0)$ as in\/ \eqref{eq:typversal},
  and a function\/ $\sigma:(\R^4,\zv)\to(\R,0)$ satisfying\/
  $\partial\sigma/\partial t(\zv)\ne0$,
  there exists a 
  diffeomorphism-germ\/ $\Theta:(\R^4,\zv)\to(\R^4,\zv)$
  such that\/
  $\Theta({\mathcal D}_{\mathcal V})=
  {\mathcal D}_{\mathcal V}$ and\/
  $\sigma\circ\Theta(t,x,y,z)=t$.
 \end{lemma}
 Let us continue the proof of Lemma \ref{lem:key}.
 Applying Lemma \ref{lem:zakaone} to $\sigma=t-g_1(x,y,z)$,
 there exists a diffeomorphism-germ $\Theta:\R^4\to\R^4$
 such that
 $\Theta({\mathcal D}_{{\mathcal V}})={\mathcal D}_{{\mathcal V}}$
 and
 $$
  (t-g_1(x,y,z))\circ\Theta=t.$$
 Let $\Psi:(\R^4,\zv)\to(\R^4,\zv)$ be a
 diffeomorphism-germ defined by $$
 \Psi(t,x,y,z)=(t-g_1(x,y,z),x,y,z).
 $$
 Then we have
 $
 {\mathcal V}\circ\Psi=\overline{G}.
 $
 We also define a diffeomorphism-germ by
 $\tilde{\Theta}=\Psi\circ\Theta$,
 and then 
 it holds that
 $$
 \tilde{\Theta}({\mathcal D}_{{\mathcal V}})
 =
 \Psi\circ\Theta({\mathcal D}_{{\mathcal V}})
 =
 \Psi({\mathcal D}_{{\mathcal V}})
 =
 {\mathcal D}_{\overline{G}}.
 $$
 Hence ${\mathcal D}_{\mathcal V}$ and ${\mathcal D}_{\overline{G}}$ 
 are diffeomorphic.
 On the other hand, defining the projection
 $\pi:\R^4\to\R$ as
 $\pi(t,x,y,z)=t$,
 we have
 $$
 \pi\circ\tilde{\Theta}
 =
 \pi\circ\Psi\circ\Theta
 =
 (t-g_1(x,y,z))\circ\Theta
 =
 t.
 $$
 Thus it holds that $\pi\circ\tilde{\Theta}=\pi$.
 Hence for each $t$, it holds that ${\mathcal D}_{\mathcal V}\cap\{t=0\}$
 and ${\mathcal D}_{\overline{G}}\cap\{t=0\}$ are also diffeomorphic.
 Since 
 $
 {\mathcal D}_{\overline{G}}\cap\{t=0\}
 =
 {\mathcal D}_{G}
 $
 and both the regular sets 
 of ${\mathcal D}_{\mathcal V}\cap\{t=0\}$ and ${\mathcal D}_{G}$
 are dense in $(\R^2,\zv)$,
 it follows
 that ${\mathcal V}(u,v,x,y,z,0)={\mathcal V}_0(u,v,x,y,z)$ and
 $G(u,v,x,y,z)$ 
 are $P$-${\mathcal K}$ equivalent.
\end{proof}

 Here, we calculate the discriminant set of ${\mathcal V}_0$:
 $$
 {\mathcal D}_{{\mathcal V}_0}=
 \{(x,y,z);
 x=-3u^2\mp v^2,\,y=\mp 2uv,\,z=2u^3\pm 2uv^2\}.
 $$
 This is 
 a parameterization of a $D_4^\pm$ singularity of a
 front.
By the same arguments as in Section 3,
for proving Theorem \ref{thm:r3} it suffies that we
construct a function $\phi$ and an unfolding $\Phi$
satisfying the conditions in Lemma \ref{lem:key}.

  \subsection{Unfolding of a given front}
Let $(f,\nu):\R^2\to\R^3$
be a front satisfying the conditions (a) and (b) of
Theorem \ref{thm:r3}.
Consider the maps
$$
\begin{array}{rcl}
\Phi(u,v,x,y,z)
&=&
\inner{\big(f_1(u,v),f_2(u,v),f_3(u,v)\big)-(x,y,z)}{\nu(u,v)}\\
 &=&
  \inner{(f_1-x,f_2-y,f_3-z)}{\nu},\\
\phi(u,v)&=&\Phi(u,v,0,0,0).
\end{array}
$$
By the same argument as in the case of $\R^4$, 
we see that the discriminant set ${\mathcal D}_\Phi$ coincides 
with the image of $f$.
Again by the same calculation as in the case of $\R^4$,
we can show that
$\phi$ is right equivalent to $u^3+ uv^2$ \resp{$u^3-uv^2$}
if and only if
$\det(\Hess\lambda(\zv))<0$ (resp. $\det(\Hess\lambda(\zv))>0$).
Since $f$ is a front,
an unfolding $\Phi(u,v,x,y,z)$ of $\phi$ 
is a Morse family of hypersurfaces.
 By the condition for the determinant of the Hessian matrix,
 the regular set of $f$ is dense in $(\R^2,\zv)$.
 Thus the regular set of ${\mathcal D}_\Phi$ is also dense.
 Hence by Lemma \ref{lem:key},
 $\Phi(u,v,x,y,z)$ is $P$-${\mathcal K}$-equivalent to ${\mathcal V}_0$.
 This proves Theorem \ref{thm:r3}.
\section{Examples}
\label{sing-curv}
Here we give two examples where the criteria
for typical $D_4^\pm$ singularities appear.
Let us consider the map
$
(u,v)\mapsto(uv, u^2+3v^2,u^2v+v^3).
$
This has a $D_4^+$ singularity at $\zv$.
Set $\nu=(2u,v,-2)/\delta$ $(\delta=(4u^2+v^2+4)^{1/2})$.
Then by \eqref{eq:signed-area-density}, we have
$
\lambda=(4u^2+4u^4-12v^2-11u^2v^2-3v^4)/\delta.
$
Thus we have $d\lambda(\zv)=\zv$ and
$$
\det(\Hess\lambda(\zv))=
\frac{1}{4}\det\pmt{8&0\\0&-24}<0.
$$

Now consider the map 
$
(u,v)\mapsto(uv, u^2-3v^2,u^2v-v^3).
$
This has a $D_4^-$ singularity.
Set $\nu=(2u,v,-2)/\delta$. Then we have
$
\lambda=(4u^2+4u^4+12v^2+13u^2v^2+3v^4)/\delta.
$
Thus we also have $d\lambda(\zv)=\zv$ and
$$
\det(\Hess\lambda(\zv))
=\frac{1}{4}\det\pmt{8&0\\0&24}>0.
$$
\section{Application}
In this section,
as an application of Theorem \ref{thm:r3},
we study the singular curvature of
cuspidal edges near a $D_4^+$ singularity in $\R^3$.
First, we give a brief review of the singular curvature of
cuspidal edges, as given in \cite{front}.
Let $f:(\R^2,\zv)\to(\R^3,\zv)$ be a cuspidal edge and
$\nu$ its unit normal vector.
Then 
there exists a regular curve $\gamma(t)$ passing through $\zv$.
Furthermore, we have a non-vanishing
vector field $\eta(t)$ along $\gamma(t)$
such that $\eta(t)\in\ker df_{\gamma(t)}$ and 
$\bigl(\gamma'(t),\eta(t)\bigr)$ is a positively oriented frame field
of $\R^2$ along $\gamma$.
Then the {\em singular curvature\/} $\kappa_s(t)$ of the cuspidal edge
$\gamma(t)$ is
defined as follows \cite{front}:
\begin{equation}\label{eq:def-singular-curvature}
 \kappa_s(t)=\sign\bigl(d\lambda(\eta)\bigr)\,
  \frac{\det\bigl(\hat \gamma'(t),\
  \hat \gamma''(t),\
  \nu\circ\gamma(t)\bigr)}
  {|\hat \gamma'(t)|^3},
\end{equation}
where $\hat \gamma(t)=f(\gamma(t))$ and ${}'=d/dt$.
For the geometric meaning of the singular curvature,
see \cite{front}.

There are four curves emanating from a $D_4^+$ singularity,
each consisting of a cuspidal edge
(see Figure \ref{fig:dfour}).
We study the properties of the singular curvatures of
these four curves.

Let $f:(\R^2,\zv)\to(\R^3,\zv)$ be a $D_4^+$ singularity.
Then by Theorem \ref{thm:r3},
we have a regular curve
$\gamma:((-\ep,\ep),0)\to(\R^2,\zv)$ such that
$\image\gamma\subset S(f)$.
Set $\hat\gamma(t)=f(\gamma(t))$.
Then the following proposition holds.
\begin{proposition}
 If
 $$
 \det(\hat\gamma''(0),\hat\gamma'''(0),\nu(\zv))\ne0
 $$
 holds, then the singular 
 curvature\/ $\kappa_s(t)$ of\/ $\gamma(t)$,
 approaching\/ $\gamma(0)$ from both sides,
 diverges. Moreover,
 it diverges with opposite sign on opposite sides.
\end{proposition}
\begin{proof}
 Since $\hat\gamma$ at $0$ is right-left equivalent to
 the germ $(t^2,t^3,0)$ at $t=0$,
 the signs of
 $$
 \lim_{t\to+0}\det(\hat\gamma',\hat\gamma'',\nu(\gamma))
 ,\qquad
 \lim_{t\to-0}\det(\hat\gamma',\hat\gamma'',\nu(\gamma))
 $$
 are the same.
 On the other hand, by Theorem \ref{thm:r3},
 $\sign(d\lambda(\eta))$ when $t>0$ and when $t<0$
 are opposite to each other (see Figure \ref{fig:neard4}).
 Applying L'Hospital's rule to
 the formula (\ref{eq:def-singular-curvature}) twice,
 we have the conclusion.
\begin{figure}[htbp]
 \begin{center}
 \begin{picture}(298,109)(0,0)
  \put(98,72){\makebox{$\eta$}}%
  \put(220,30){\makebox{$\gamma(t)$}}%
  \put(148,37){\makebox{$\zv$}}%
  \put(80,22){\makebox{$S(f)$}}%
  \put(110,12){\makebox{$\lambda>0$}}%
  \put(230,72){\makebox{$\lambda>0$}}%
  \put(230,12){\makebox{$\lambda<0$}}%
  \put(110,72){\makebox{$\lambda<0$}}%
  \put(50,0){\includegraphics[width=.5\linewidth, bb=0 0 298 109]{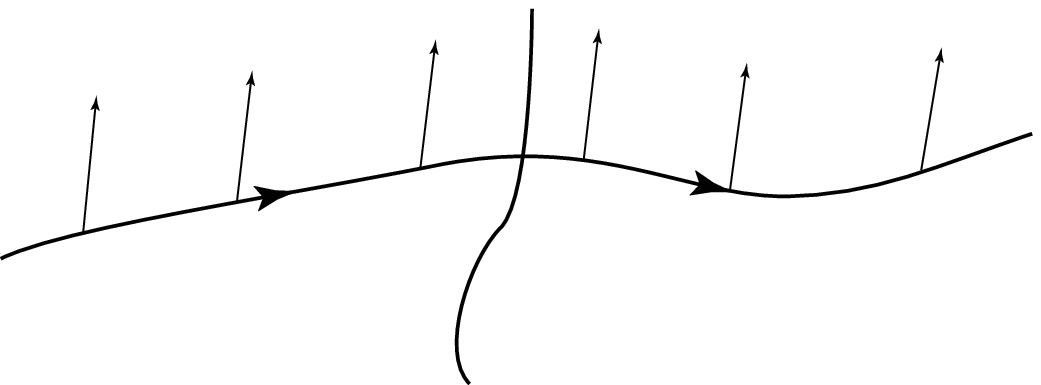}}
 \end{picture}
  \caption{Singular set and the null vector field.}
   \label{fig:neard4}
 \end{center}
\end{figure}
\end{proof}
Here we give an example applying this proposition:
\begin{example}
 The front-germ
  $f(u,v)=(u v,u^2+3 v^2,u^2 (1+v)+v^2 (3+v))$ at $\zv$ has a
 $D_4^+$ singularity.
 The front $f$  has two 
 curves $\gamma_{\pm}(t)=(\pm\sqrt{3}t,t)$
 passing through $\zv$, consisting of cuspidal edges.
 Set $\gamma(t)=\gamma_{+}(t)$ and
 $\hat\gamma(t)=f(\gamma(t))$.
 Then we have
 $\det(\hat\gamma''(0),\hat\gamma'''(0),\nu(\zv))=-24\sqrt{6}\ne0$.
 The singular curvature is calculated as follows:
 $$
 \kappa_s(t)=\sign(t)\frac{t^2 (2-11 t)}{|t^2 (25+24 t+12 t^2)|^{3/2}}.
 $$
 Thus $\lim_{t\to+0}\kappa_s(t)=+\infty$ and
 $\lim_{t\to-0}\kappa_s(t)=-\infty$
 hold (see Figure \ref{fig:d4pks}).
\end{example}
\begin{figure}[htbp]
 \includegraphics[width=.5\linewidth,bb=0 0 476 204]{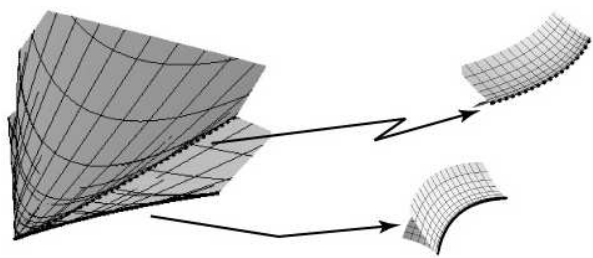}
\caption{The shapes of cuspidal edges near
 the $D_4^+$ singularity of $f$.}
 \label{fig:d4pks}
\end{figure}
\mycomment{
\begin{figure}[htbp]
 \begin{center}
 \begin{picture}(240,120)(0,0)
  \put(95,110){\makebox{$\kappa_s$}}%
  \put(98,75){\makebox{$\zv$}}%
  \put(220,75){\makebox{$t$}}%
  \put(0,0){\includegraphics[width=.5\linewidth, bb=0 0 360 180]
  {singcurveexgraph.eps}}
  \label{fig:singcurveexgraph}
   \end{picture}
  \caption{Behavior of singular curvature near
  the $D_4^+$ singularity of $f$}
 \end{center}
\end{figure}
}

\medskip
\begin{flushright}
 \begin{tabular}{l}
  Department of Mathematics,\\
  Faculty of Education, Gifu University\\
  Yanagido 1-1, Gifu, 501-1193, Japan\\
  e-mail: {\tt ksajiO\!\!\!agifu-u.ac.jp}
\end{tabular}\end{flushright}
\end{document}